\documentclass[12pt]{article}   
\usepackage{amsmath,amssymb}  
\usepackage{amsthm}

\usepackage{verbatim}
\usepackage{amsfonts}
\usepackage{amssymb}
\usepackage{mathrsfs}
\usepackage[dvips]{graphicx}
\usepackage{pstricks} 
\usepackage{pst-node}
\usepackage{todonotes}

\usepackage{setspace}
\usepackage[inner=2.4cm,outer=4cm,top=2.4cm,bottom=2.4cm,marginparwidth=90pt]{geometry}   
\usepackage{times}

\usepackage[notref,notcite,final]{showkeys}  
  
\numberwithin{equation}{section}

\usepackage{framed}  

\newcommand{\bigR}{{\mathbb R}}

\newcommand{\itemEq}[1]{%
         \begingroup%
         \setlength{\abovedisplayskip}{0pt}%
         \setlength{\belowdisplayskip}{0pt}%
         \parbox[c]{\linewidth}{\begin{flalign}#1&&\end{flalign}}%
         \endgroup}

\newtheorem{theorem}{Theorem}[section]  
\newtheorem{lemma}[theorem]{Lemma}  
\newtheorem{proposition}[theorem]{Proposition}  
\newtheorem{corollary}[theorem]{Corollary}  
\newtheorem{remark}[theorem]{Remark}

\allowdisplaybreaks[1]  
 
\title{Singularity formation in axially symmetric mean curvature flow with Neumann boundary}  

\author{John Head, Sevvandi Kandanaarachchi}  

\date{}

\begin{document}

\maketitle

 \onehalfspacing  
  

\abstract
We study mean curvature flow of smooth, axially symmetric surfaces in $\mathbb{R}^3$ with Neumann boundary data. We show that all singularities at the first singular time must be of type I.

\section{Introduction}

Consider a smooth, $n$-dimensional hypersurface immersion $\mathbf{x}_0: M^n \rightarrow \mathbb{R}^{n+1}$. The solution of mean curvature flow generated by $\mathbf{x}_0(M^n)$ is the one-parameter family $\mathbf{x}: M^n \times [0,T) \rightarrow \mathbb{R}^{n+1}$ of smooth immersions satisfying
\begin{equation}\label{eq:int_2}
\frac{\partial}{\partial t} \mathbf{x}(l, t)=-H(l,t)\nu(l,t), \hspace{5 mm} l\in M^n, t\geq0 \, ,
\end{equation}
\noindent
with $\mathbf{x}(\cdot,0) = \mathbf{x}_0$. Here $\nu(l,t)$ represents a choice of unit normal -- the outward-pointing unit normal in the closed setting -- and $H(l,t)$ is the mean curvature. According to our choice of signs, the right-hand side is the mean curvature vector and the mean curvature of the round sphere is positive. We henceforth write $M_t = \mathbf{x}(\cdot, t)(M^n)$.\\

\noindent
In \cite{GH1} Huisken initiated a formal investigation of the classical evolution (\ref{eq:int_2}), establishing that any compact and uniformly convex hypersurface of dimension at least two must contract smoothly to a point in finite time and in an asymptotically round fashion. \\

\noindent
Altschuler, Angenent and Giga \cite{AAG} studied generalized viscosity solutions of mean curvature flow in the axially symmetric setting. They showed in particular that there is a finite set of singular times outside of which the evolving hypersurfaces are smooth. In addition, they showed that at each of the singular times, a finite number of necks pinch off at isolated points along the axis of rotation (or else the entire connected component of the hypersurface shrinks to a point on the axis of rotation). The flow therefore produces a smooth family of smooth hypersurfaces away from the axis of rotation. \\

\noindent
In the setting of closed two-convex hypersurfaces of dimension at least three, Huisken and Sinestrari \cite{HS3} introduced a surgery-based algorithm for modifying high curvature regions in a topologically controlled way, thereby constructing a non-canonical continuation of the classical evolution which is compatible with the well-established theory of weak solutions in a precise quantitative sense. \\

\noindent
In this paper we consider a smooth, compact, 2-dimensional hypersurface $M_0$ in $\mathbb{R}^3$ with boundary $\partial M_0 \neq \emptyset$. We assume that $M_0$ is smoothly embedded in the domain
$$ G = \{ \mathbf{x} \in \bigR^{3} : a < x_1 < b \} \, , \hspace{3mm}  0<a<b \, ,
$$
and that the free boundary satisfies the constraint $\partial M_0 \subset \partial G$. Moreover, we assume that $M_0$ is axially symmetric and that the surface meets the planes $x_1 = a$ and $x_1 = b$ orthogonally. \\

\noindent
In this setting, the authors \cite{JHSK1} proved that if the mean curvature is uniformly bounded on any finite time interval, then no singularity can develop during that time. \\

\noindent
Of course, the well-known comparison principle guarantees that singularities must develop in finite time, motivating an analysis of the types of singularities that can occur. In \cite{GH2} Huisken showed that if $M_0$ has positive mean curvature, then all singularities must be of type I. Moreover, they behave asymptotically like shrinking cylinders after appropriate parabolic rescaling. \\

\noindent 
In this paper we obtain a complete classification of singularities without any restriction on the mean curvature of the initial data. We emphasize that we henceforth restrict our attention to 2-dimensional surfaces in $\mathbb{R}^3$.

\begin{theorem}{\emph{(Singularity Classification)}}\label{Thm:MainThm}
Consider a smooth, axially symmetric solution $M_t$ of mean curvature flow \eqref{eq:int_2} in $\mathbb{R}^3$ with Neumann boundary on the maximal time interval $[0,T)$, where $T>0$ denotes the first singular time. Then all singularities that develop as $t\to T$ must be of type I.
\end{theorem}

\noindent
It follows as an immediate consequence of Theorem \ref{Thm:MainThm} that any axially symmetric surface with Neumann boundary cannot have $H<0$ everywhere. This property is independent of mean curvature flow. \\

\noindent
If there exists some $0<t_0<T$ such that $H(l,t)>0$ for all $l\in M^2$ and $t>t_0$, then Theorem \ref{Thm:MainThm} follows from the work of Huisken, see section 5 of \cite{GH2}. Our proof covers the cases in which points of negative mean curvature persist up to the singular time $T$. \\

\noindent
\textbf{Outline.} The results in this paper are organised as follows. In section \ref{Section_Notation} we establish notation and introduce the requisite definitions.  Section~\ref{sec:AprioriEstimates} contains preliminary height, gradient and curvature estimates. In section~\ref{Section_Negative_MC} we use these \textit{a priori} estimates to prove directly that no singularities can develop in regions of negative mean curvature. \\

\noindent
In section~\ref{sec:rescaling} we recall the parabolic rescaling techniques adopted in \cite{JHSK1}.  Section~\ref{sec:bddMC} uses this rescaling procedure to rule out singularities in certain regions of the hypersurface. Finally, in section~\ref{sec:boundedAonH}, we combine these results with the work of Huisken in \cite{GH2} to establish that all singularities must be of type I, completing the proof of the main theorem. \\

\noindent
We point out that the estimates in section~\ref{sec:boundedAonH} rely on parabolic maximum principles for non-cylindrical domains. The results employed in this section go somewhat beyond standard theory and have therefore been included in an Appendix. We refer the reader to \cite{MASK2} for further details.

\section{Notation and preliminaries}\label{Section_Notation}
In this paper we follow the notation used in \cite{JHSK1}. This agrees in particular with the notation used by Huisken in \cite{GH2} and by Athanassenas in \cite{MA1}. \\

\noindent
Let $\rho_0:[a,b] \rightarrow  \mathbb{R} $ be a smooth, positive function on the bounded interval $[a,b]$ with $\rho'_0(a) = \rho'_0(b)=0 $. Consider the surface $M_0$ in $\mathbb{R}^{3} $ obtained by rotating the graph of $\rho_0$ around the $x_1$-axis. We evolve $M_0$ according to \eqref{eq:int_2} with Neumann boundary conditions at $x_1 = a$ and $x_1 = b$. Equivalently, we can consider the evolution of a periodic surface defined on the entire $x_1$ axis. This deformation process preserves axial symmetry. We denote by $T>0$ the extinction time of the smooth evolution. \\

\noindent
Let $\mathbf{i}_1, \mathbf{i}_2, \mathbf{i}_{3}$ be the standard basis vectors in $\mathbb{R}^{3}$ associated with the $x_1, x_2, x_{3}$ axes respectively. We introduce a local orthonormal frame  $\tau_1(t), \tau_2(t) $ on the evolving surfaces $M_t$ such that 
\[ \left\langle\tau_2(t), \mathbf{i}_1\right\rangle = 0,  \hspace{5mm} \text{and} \hspace{5mm} \left\langle\tau_1(t), \mathbf{i}_1\right\rangle > 0  \, .\]

\noindent
Let $ \omega= \frac{\hat{\mathbf{x}}}{|\hat{\mathbf{x}}|} \in \mathbb{R}^{3}$ be the outward-pointing unit normal to the cylinder intersecting $M_t$ at the point $\mathbf{x}(l,t)$. Here $\hat{\mathbf{x}}= \mathbf{x} - \left\langle \mathbf{x}, \mathbf{i}_1 \right\rangle \mathbf{i}_1$.  We additionally define 

$$y = \left\langle \mathbf{x}, \omega \right\rangle \hspace{3 mm} \text{and}  \hspace{3mm} v=\left\langle \omega, \nu \right\rangle ^{-1}\, .$$ 
Following convention we call $y$ the { \it height function} and $v$ the {\it gradient function}. We emphasize that $\rho: [a, b] \times [0, T) \rightarrow \bigR$, whereas $y: M^2 \times [0, T) \rightarrow \bigR $. Note also that $v$ is a geometric quantity related to the inclination angle. More precisely, $v$ corresponds to $\sqrt{1+ \rho'^2}$ in our setting. \\

\noindent
We denote by $g=\{g_{ij}\}$ the induced metric and by $A=\{h_{ij}\}$ the second fundamental form at the space-time point $(l,t)\in M^2\times [0,T)$. Following \cite{GH2}, we define the quantities
\begin{equation}\label{notEq:1.1}
 p =  \left\langle\tau_1, \mathbf{i}_1 \right\rangle y^{-1} \quad \mbox{and} \quad  q= \left\langle\nu, \mathbf{i}_1 \right\rangle y^{-1}, 
\end{equation}

\noindent
which satisfy
\begin{equation}\label{notEq:1.2}
p^2 + q^2 = y^{-2} \, .
\end{equation}

\noindent
The second fundamental form has eigenvalues
$$p = \frac{1}{\rho \sqrt{1+ \rho'^2}}$$
and
\[ k= \left\langle \overline{\nabla}_{1} \nu, \tau_1 \right\rangle = \frac{-\rho''}{(1+\rho'^2)^{3/2}}. \]

\noindent
We recall the following evolution equations, see \cite{JBT1,GH1}.

\begin{lemma}\label{Lemma_Evolution_Equations_2}{\emph{(Evolution Equations)}} We have the evolution equations: 
\begin{itemize}
\item[(i)] $\frac{\partial}{\partial t} y = \Delta y - \frac{1}{y}  \, ; $
\item[(ii)]$ \frac{\partial}{\partial t} v = \Delta v  -|A|^2v + \frac{v}{y^2} - \frac{2}{v}|\nabla v|^2\, ;  $
\item [(iii)] $ \frac{\partial}{\partial t} k = \Delta k + |A|^2k -2q^2(k-p)  \, ;  $
\item [(iv)] $ \frac{\partial}{\partial t} p = \Delta p + |A|^2p + 2q^2(k-p)  \, ;  $
\item [(v)] $ \frac{\partial}{\partial t} q = \Delta q + |A|^2q + q\left(p^2 - q^2 - 2kp \right)  \, ;  $
\item [(vi)]$ \frac{\partial}{\partial t} H = \Delta H + H  |A|^2 \, . $
\end{itemize}

\end{lemma}

\noindent 
Finally, we establish notation for the smooth space-time hypersurface
$$\Omega :=\bigcup\limits_{0\leq t<T}M_t\times \{t\} \subset\mathbb{R}^3\times\mathbb{R}^+ \, .$$
Let $c>0 $. For each $0\leq t <T$ we define
\[
\Omega_t^-:=\{ \mathbf{x}(l,t)\in M_t : \, H(l,t)< -c  \}\subset M_t.
\]
We let $\Omega^-:=\cup_{t<T}\Omega_t^-\times\{t\} \subset \Omega$ and we denote by $\Gamma_{ \Omega^-}$ the non-cylindrical parabolic boundary of $\Omega^-$. \\

\section{A priori estimates}\label{sec:AprioriEstimates}

\noindent
We establish \textit{a priori} height, gradient and curvature estimates. In the first step, we show that the height function $y$ has a lower bound in $\Omega^-$. 

\begin{lemma}\label{Lemma_yLowerBound} {\emph{(Height Bound)}} There exists a constant $c  >0 $ depending on $M_0$ and $\Gamma_{\Omega^-}$ such that
\[
\inf_{ \Omega^-} y = \inf_{\Gamma_{\Omega^-}} y \geq c . 
\]
\end{lemma}

\begin{proof}
In $\Omega^- $ we have
\[
 \frac{dy}{dt} = -Hpy > 0 .
\]
That is to say, the height function is increasing in $\Omega^-$, from which we deduce that
\[ \inf_{ \Omega^- } y = \inf_{ \Gamma_{\Omega^-}} y \, . \]
\noindent
Now suppose that $y$ reaches zero on $\Gamma_{\Omega^-}$ at some time $t_*$. In particular, it must therefore have decreased immediately before $t_*$. Consider a constant $0<\bar{c}<c$ and the corresponding domain $\bar{\Omega}^-\supset \Omega^-$. On this new domain we once again have $\frac{dy}{dt} > 0 $, yielding a contradiction. 
\end{proof}

\begin{lemma} \label{Lemma_UV} {\emph{(Gradient Estimate)}}
There exists a constant $c>0$ depending only on the initial hypersurface $M_0$ such that $ yv \leq  c$.
\end{lemma}

\begin{proof}
 See Lemma 5.2 in \cite{JHSK1}.
\end{proof}

\noindent
\noindent
In particular, Lemma \ref{Lemma_UV} controls the gradient function away from the axis of rotation. This result provides the first indication that type II singularities can't develop in our setting.  
Combining this with Lemma \ref{Lemma_yLowerBound} we can therefore find a constant $c >0$ depending on $M_0$ and $\Gamma_{\Omega^-}$ such that $v\vert_{\Omega^-} \leq c \,.$

\begin{proposition} \label{Prop_KonP} {\emph{(Curvature Estimate)}}
There is a constant $c>0$ depending only on the initial hypersurface $M_0$ such that  $\frac{k}{p} \leq c$.
\end{proposition}

\begin{proof}
 See section 5 of \cite{GH2}.
\end{proof}


\noindent
The next result is a generalisation of Proposition 5.4 in \cite{JHSK1}. For the convenience of the reader, we include the full proof below.

\begin{proposition} \label{Prop_KonP2} {\emph{(Absolute Curvature Estimate)}}
Let $c>0$. Suppose that $l_0\in M^2$ and $t_0\in [0,T)$ are such that $H(l_0,t_0)\geq 0$ or $|H(l_0,t_0)|\leq c$. There exists a constant $C>0$ depending only on $c$ and the initial hypersurface $M_0$ such that $\frac{|k(l_0,t_0)|}{p(l_0,t_0)} \leq C$. 
\end{proposition}  

\begin{proof}
If both $H$ and $k$ are positive, Proposition \ref{Prop_KonP} yields the pointwise estimate
\[
\frac{|k|}{p} = \frac{k}{p} \leq c. 
\]
If $H \geq 0$ and $k <0$ , then from $k + p \geq 0$ we obtain $ -|k| + p \geq 0 $ and
\[
 \frac{|k|}{p} \leq 1\, .
\]
If $|H| \leq  c $ and $k>0$ then $H>0$ and the result holds, so it remains to consider $|H| \leq c$ and $k <0 $. We have $- |k| + p \geq -c$ so that
\[
\frac{|k|}{p} \leq  1 +  \frac{c}{p}.
\]
From Lemma \ref{Lemma_UV} we have $\frac{1}{p} = vy \leq  c$. This completes the proof. 
\end{proof}

\section{Negative mean curvature} \label{Section_Negative_MC}

We use direct \textit{a priori} estimates to establish that no singularities can develop in regions of negative mean curvature. This section has some overlap with parts of \cite{MASK2}, which studies the first singular time for volume preserving mean curvature flow. \\

\noindent
Let $\tilde{c}_0>0 $ and consider the corresponding sets $\tilde{\Omega}_t^-\subset M_t$ such that $H<-\tilde{c}_0$. In addition, we define $\tilde{\Omega}^-:=\cup_{t<T}\tilde{\Omega}_t^-\times\{t\} $. \\

\noindent
The following result is a generalisation of Proposition 4.6 in \cite{MASK2}. 

\begin{proposition}\label{Prop_AsquaredBddAwayFromAxisNew} {\emph{(Curvature away from Axis)}}
There exists a constant $c>0$ such that $|A|^2 \leq c$ in $\tilde{\Omega}^-$.
\end{proposition}
\begin{proof}
Consider the product $g = |A|^2 \varphi(v^2)$, where $\varphi(r) = \frac{r}{\lambda - \mu r}$ and $\lambda, \mu >0 $ are free constants. The evolution equation for $g$ yields the estimate
\[\left ( \frac{d}{dt}- \Delta \right) g \leq -2\mu g^2 -2\lambda \varphi v^{-3} \left\langle \nabla v \, , \nabla g \right\rangle - \frac{2\lambda \mu}{(\lambda-\mu v^2)^2}|\nabla v|^2 g 
+ \frac{2(n-1)}{y^2}v^2 \varphi'|A|^2  \, .\]

\noindent
Following Proposition 6.2 of \cite{AthKan1} with $\mu>\frac{3}{4}$ and $\lambda>\mu \max v^2$ we can find a constant $c>0$ depending on $\mu$, $\lambda$, $\tilde{c}_0$ and $M_0$ such that
\begin{equation}\label{eq_Ev_Eq_1.43}
 |A|^2\varphi(v^2)  \leq \max \left( \max_{\tilde{\Omega}^-_0} |A|^2\varphi(v^2), \hspace{2mm}  \max_{\begin{subarray}{l} {\partial \tilde{\Omega}^-_t}\\ t<T \end{subarray}} |A|^2\varphi(v^2), \hspace{2mm} c \right) \, .
\end{equation}
\noindent
Note that by construction we have $|H| =\tilde{c}_0 $ on $\partial \tilde{\Omega}^-_t$ for all $0<t<T$. In \cite{MASK2}, there is a positive mean curvature restriction on the boundary. Proposition \ref{Prop_KonP2} now yields a constant $c >0$ depending only on $\tilde{c}_0$ and $M_0$ such that on $\partial \tilde{\Omega}^-_t$ we have $\frac{|k|}{p} \leq c$ for all $t<T$. On $\partial \tilde{\Omega}^-_t$ we have
\begin{align*}
 |A|^2 &= k^2 + p^2 
 \leq (1+c^2)p^2  
  \leq (1 + c^2) y^{-2} \leq c  
 \end{align*}
\noindent
for all $t<T$. The final estimate follows from Lemma~\ref{Lemma_yLowerBound}.  Since $v$ is bounded in $\tilde{\Omega}^-$, $\varphi(v^2)$ is bounded from above.
The product $|A|^2\varphi(v^2) $ is therefore bounded on $\Gamma_{\tilde{\Omega}^-}$ and on $\tilde{\Omega}^-$ courtesy of (\ref{eq_Ev_Eq_1.43}).  By our choice of $\lambda$, and since $v \geq 1 $, we have a bound on $(\varphi(v^2))^{-1}$. This completes the proof.\\
%
\end{proof}

\noindent
As in Proposition 4.8 of \cite{MASK2} we obtain as a consequence that $H$ is bounded from below in $\tilde{\Omega}^-$. That is, $H$ cannot go to $-\infty$.

\begin{corollary}\label{Lemma_HLowerBound} {\emph{(Mean Curvature Bound)}}
There exists a constant $c>0$ independent of time such that  $H(l, t) \geq -c$ for all $l \in \ M^2$ and $t\in [0,T)$. 
\end{corollary}
\begin{proof}
\noindent
By construction we have $H<0$ in $\tilde{\Omega}^-$ and $H\geq -\tilde{c}_0$ on $\Omega\setminus\tilde{\Omega}^-$. Using Proposition \ref{Prop_AsquaredBddAwayFromAxisNew}, we can find a constant $c>0$ such that $|A|^2 \leq c $ in $\tilde{\Omega}^-$.  The trivial inequality $H^2 \leq 2|A|^2$ completes the proof.
\end{proof}
\noindent
As a result of Proposition~\ref{Prop_AsquaredBddAwayFromAxisNew},
singularities can only develop in $\Omega\setminus\tilde{\Omega}^-$, which we investigate in the remaining sections.

\section{Rescaling}\label{sec:rescaling}

We established in Section \ref{Section_Negative_MC} that singularities are restricted to $\Omega\setminus\tilde{\Omega}^-$. It is well-known from \cite{AAG, MA2} that singularities of axially symmetric mean curvature flow are finite and discrete.
Let $T$ be the first singular time for the smooth evolution and let $(x_*,T) \in \Omega \subset \mathbb{R}^3\times \mathbb{R}^+$ be a singular point in space-time. It is therefore possible to analyse a space-time neighbourhood $N_{\epsilon} \subset \mathbb{R}^3\times \mathbb{R}^+$ centred at $(x_*,T)$ such that the flow is smooth inside $N_\epsilon \backslash (x_*,T)$. There are three possible cases:\\


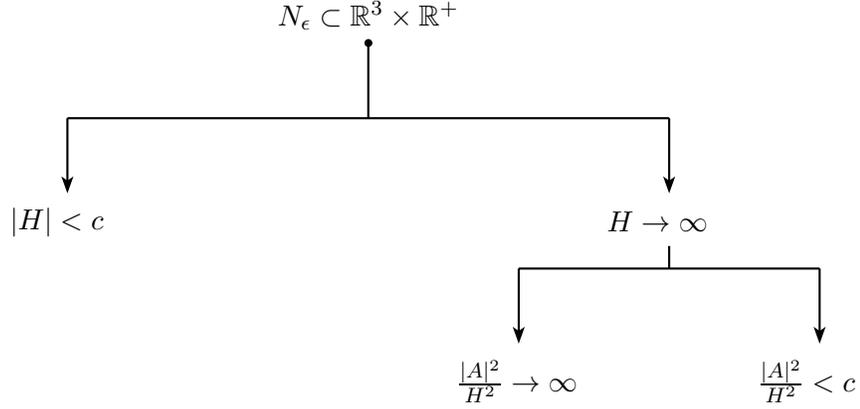
\psset{xunit=1.0cm,yunit=1.0cm,dotstyle=o,dotsize=3pt 0,linewidth=0.8pt,arrowsize=3pt 2,arrowinset=0.25}
\begin{pspicture*}(-4.3,-0.5)(9.78,6.3)
\psline(1,5)(1,4)
\psline(-3,4)(5,4)
\psline{->}(-3,4)(-3,3)
\psline{->}(5,4)(5,3)
\psdots[dotstyle=*,linecolor=black](1,5)
\uput[90](1, 5){$N_\epsilon \subset \mathbb{R}^3\times \mathbb{R}^+$}
\uput[0](-3.94, 2.62){$|H| < c$ }
\uput[0](4, 2.62){$H \to \infty$}
\psline(5,2.3)(5,2)
\psline(3,2)(7,2)
\psline{->}(3,2)(3,1)
\psline{->}(7,2)(7,1)
\uput[0](2, 0.5){$\frac{|A|^2}{H^2}\to \infty$ }
\uput[0](6, 0.5){$\frac{|A|^2}{H^2}<c$}
\end{pspicture*}

\noindent
Next we introduce the parabolic rescaling techniques which will be used in section \ref{sec:bddMC} to analyse the cases $|H| < c$ and $H\to\infty $ with $|A|^2/H^2 \to\infty $. In particular, we use a standard contradiction argument in section \ref{sec:bddMC} to show that a singularity cannot develop in either of these cases. Singularities can therefore only occur if $H \to \infty$ and $|A|^2/H^2$ is bounded. This remaining case is covered in section~\ref{sec:boundedAonH}. \\

\noindent
We employ the rescaling procedure introduced in \cite{JHSK1}. Consider the smooth solution $M_t$ of \eqref{eq:int_2} for $t \in [0, T)$, and suppose that a singularity forms at the centre of the space-time neighbourhood $N_{\epsilon}$ at the singular time $T$. In particular, $|A|^2 \rightarrow \infty$ as $t \rightarrow T$. For integers $i\geq 1$, we select times $t_i\in\left[0 , T - \frac{1}{i} \right]$ and points $l_i \in M^2$ such that:
\begin{enumerate}
    \item $t_i\to T$
    \item $\mathbf{x}(l_i, t_i)$ lies on the $x_1x_{3}$ plane 
    \item \itemEq{\label{eq_Hbdd_eq15}
     |A|(l_i, t_i)  = \max_{ l \in M^{2}, \,  t \leq T- \frac{1}{i} }  |A|(l,t)  }
\end{enumerate}



\noindent
We write $\alpha_i = |A|(l_i, t_i) $ and $ \mathbf{x}_i = \mathbf{x}(l_i, t_i) \, .$ Note that for $i$ sufficiently large, $\mathbf{x}_i$ is contained in $N_{\epsilon}$. We now rescale $M_t$ to obtain the family $\tilde{M}_{i, \tau}$ defined by
\begin{equation} \label{HbddEq_1}
 \tilde{\mathbf{x}}_i(l \,, \tau) = \alpha_i \left( \mathbf{x}(l \,, \alpha_i^{-2} \tau + t_i )	- \left\langle \mathbf{x}_i , \mathbf{i}_1 \right\rangle \mathbf{i}_1 \right)  \, ,
\end{equation} 
where $ \tau \in [ -\alpha_i^2 t_i, \alpha_i^2 (T- t_i - \frac{1}{i} )] \, .$
\\

\noindent
Note that we rescale from a point on the axis of rotation corresponding to the point of maximum curvature, preserving axial symmetry. We define $\tilde{\rho}_{i , \tau}$ to be the generating curves of $\tilde{M}_{i, \tau}$. We denote by $|\tilde{A}_i| $ and $\tilde{H_i}$ the second fundamental form and mean curvature of $\tilde{M}_{i, \tau} \,$, respectively. By definition
$$  \tilde{H_i}( \cdot \,, \tau ) = \alpha_i^{-1} H(\cdot \,, \alpha_i^{-2} \tau + t_i ) \,  \hspace{3mm} \text{and} \hspace{3mm} |\tilde{A}_i|( \cdot \,, \tau ) = \alpha_i^{-1}	|A| (\cdot \,, \alpha_i^{-2} \tau + t_i )  \, .$$	
For $t \leq T - \frac{1}{i}$ we have 
\begin{equation}\label{HbddEq_2}
 \alpha_i^{-1}	|A| (\cdot \,, \alpha_i^{-2} \tau + t_i ) \leq 1 \, . 
\end{equation}
\noindent
Note that
\begin{equation}\label{rescaledMCF}
\frac{\partial}{\partial \tau} \tilde{\mathbf{x}}_i  		= - \alpha_i^{-1} H \nu \, 
										= - \tilde{H}_i \nu \, .
\end{equation}
The rescaled flows cannot drift away to infinity: applying Proposition \ref{Prop_KonP2} we can find a constant $c>0$ depending only on $\tilde{c}_0$ and $M_0$ such that
$$ |A|=\sqrt{k^2+p^2}\leq c p \leq c y^{-1} . $$
After rescaling, this becomes
$$ |\tilde{A}_i| \leq c (\alpha_{i} y)^{-1} =c \tilde{y}_i^{-1} . $$

\noindent
Since $|\tilde{A}_i|(l_i,0)=1$ for all $i$, we have a bound on $\tilde{y}$ and we can therefore extract a convergent subsequence of points on the $x_3$ axis.
\\

\noindent
Along the sequence of rescalings we have the uniform curvature bound $|\tilde{A}_i|^2 \leq 1 $. Since each rescaled flow again satisfies (\ref{rescaledMCF}), this gives rise to uniform bounds on all covariant derivatives of the second fundamental form, see for example \cite{GH1}. \\

\noindent
Using the Arzela-Ascoli theorem we can therefore find a further subsequence which converges uniformly in $C^{\infty}$ on compact subsets of $\mathbb{R}^{3}\times\mathbb{R}$ to a non-empty smooth limit flow which exists on an interval $(-\infty,\beta)$ where $\beta\in[0,\infty]$. The crucial step is to analyse the properties of this limit flow, which we label $\tilde{M}_{\infty,\tau}$.

\section{No singularities}\label{sec:bddMC} 

\noindent
In this section we use a standard contradiction argument to show that no singularities can develop as long as the mean curvature remains bounded. In addition, we show that no singularity can develop if both $H\to\infty$ and $|A|^2/H^2\to\infty$. The remaining scenario is analysed in the next section.

\begin{theorem}{\emph{(Bounded Mean Curvature)}}\label{Thm:BddMC}
Consider a smooth, axially symmetric solution $M_t$ of mean curvature flow \eqref{eq:int_2} with Neumann boundary on the time interval $[0,T)$ for some $T>0$. Then no singularity can develop if $H$ remains bounded. 
\end{theorem}

\noindent
\begin{proof}
\noindent
Suppose in order to obtain a contradiction that a singularity forms at the point $x_*$ on the axis of rotation and at time $t=T$; in particular, $|A|^2 \rightarrow \infty$ as $t \rightarrow T$. We assume in addition that the mean curvature remains bounded in a space-time neighbourhood around the point $(x_*,T)$. We rescale using the procedure outlined in section \ref{sec:rescaling} and analyse the properties of the resultant limit flow $\tilde{M}_{\infty,\tau}$. Since by assumption $|H|<c$ for some $c>0$, we have
$$\lim_{i\to\infty}\tilde{H}_i=0.$$
The limit flow $\tilde{M}_{\infty,\tau}$ is a stationary solution and must therefore be the catenoid. We relabel this solution $\hat{M}$ and henceforth use a `hat' to indicate that a geometric quantity is associated with the catenoid.\\

\noindent
The catenoid is obtained by rotating $\hat{y} = c \cosh(c^{-1}\hat{x}_1) $ around the $x_1$ axis. For any $\epsilon>0$ and for any $l \in M^2$ we can find $I_0\in \mathbb{N}$ such that for any fixed $\tau_0\in (-\alpha_{I_0}^2t_{I_0},0)$ we have
$$ \hat{v}(l)\hat{y}(l) -\epsilon  \leq  \tilde{v}_i(l, \tau_0) \tilde{y}_i(l, \tau_0)  \hspace{5mm} \text{for all} \hspace{2mm} i >I_0 \, . $$

\noindent
On the catenoid, $\hat{v} = \sqrt{1 + \hat{y}'^2}= \cosh(c^{-1}\hat{x}_1)$. It therefore follows from Lemma \ref{Lemma_UV} that
\begin{equation*}
 \frac{c}{ 2\alpha_i}\left( \cosh (2c^{-1}\hat{x}_1) +1  \right) - \frac{\epsilon}{ \alpha_i} \leq c  \hspace{5mm} \text{for all} \hspace{2mm} i >I_0 \, .
\end{equation*}

\noindent
For fixed $i$, the left-hand side can be made as large as we like, yielding the desired contradiction. We can therefore find a constant $c>0$ such that $|A|^2 \leq c$ for all $t \in [0, T)$. Using standard theory, see for example \cite{GH1}, we obtain estimates on all covariant derivatives of $|A|$, allowing us to extend the flow beyond $T$. This completes the proof.  
\end{proof}

\noindent
We next consider the case in which $H\to\infty$ and $|A|^2/H^2\to\infty$.

\begin{theorem}
Consider a smooth, axially symmetric solution $M_t$ of mean curvature flow \eqref{eq:int_2} with Neumann boundary on the time interval $[0,T)$ for some $T>0$. Then no singularity can develop if both $H\to\infty$ and $|A|^2/H^2\to\infty$.
\end{theorem}

\begin{proof}
We proceed as in the proof of Theorem \ref{Thm:BddMC}: suppose in order to obtain a contradiction that a singularity forms at the point $x_*$ on the axis of rotation and at time $t=T$. We again rescale the flow. If $|A|^2/H^2 \to \infty$ then $\alpha_i^{-1}H \to 0$, once again giving us a stationary limit flow, which must be the catenoid. The rest of the proof goes through unchanged.
\end{proof}

\section{Type I singularities}\label{sec:boundedAonH}

\noindent
We prove that all singularities must be of \emph{type I}:

\begin{proposition}\label{Prop_TypeI} {\emph{(Type I Singularities)}}
Consider a smooth, axially symmetric solution $M_t$ of mean curvature flow \eqref{eq:int_2} in $\mathbb{R}^3$ with Neumann boundary data on the maximal time interval $[0,T)$ for some $T>0$. Suppose that a singularity forms on the axis of rotation at $x_*\in\mathbb{R}^3$ at time $T$. Assume in addition that there exists a neighbourhood $N_\epsilon\subset \mathbb{R}^3\times \mathbb{R}^+$ centred at $(x_*,T)$ and a constant $c_0>0$ such that $|A|^2/H^2\leq c_0$ in $N_\epsilon$. Then there exists a constant $C>0$ such that 
$$ \max_{M_t\cap N_\epsilon} |A|^2 \leq C \frac{1}{T-t}  $$
for all $t<T$.
\end{proposition}

\noindent
We proceed as in section 5 of \cite{GH2}. We emphasize that our setting will generate additional boundary terms.

\begin{proof}
\noindent 
From Lemma \ref{Lemma_Evolution_Equations_2} we have
$$ \frac{\partial}{\partial t} \left(\frac{q}{H} \right) = \Delta \left(\frac{q}{H} \right) +  \frac{2}{H} \nabla_i H \nabla_i \left( \frac{q}{H} \right) + \frac{q}{H} \left(p^2 - q^2 - 2kp \right) . $$
Following \cite{GH2} we have bounds on the final term and deal with our different boundary terms by applying the non-cylindrical maximum principle, Proposition \ref{Prop_Max_Principles}, to obtain
$$ \frac{|q|}{H} \leq \max \frac{|q|}{H}\bigg\vert_{\Gamma_{N_\epsilon}} \, . $$

\noindent
Now note that 
$\vert q \vert = \left\vert \langle \nu, \mathbf{i}_1 \rangle y^{-1} \right\vert \leq y^{-1}$. By assumption, $|A|^2/H^2$ is bounded in $N_\epsilon$ so we can find a constant $c>0$ such that $H\vert_{\Gamma_{N_\epsilon}} \geq c$. In addition, it is well-known (see for example Lemma 5.2 in \cite{AAG}) that $y$ is bounded from below away from the singular point. In particular therefore we have a constant $c>0$ such that $y\vert_{\Gamma_{N_\epsilon}} \geq c $, giving us a bound on $|q|/H$.
Applying Proposition \ref{Prop_KonP} we find
$$ |q| \leq cH \leq c\left( p + k \right) \leq cp $$
in $N_\epsilon$. The rest of the proof of Proposition 5.3 in \cite{GH2} goes through unchanged in $N_\epsilon$.
\end{proof}

\noindent
Together with Theorem 6.1 and Theorem 6.2, this gives Theorem 1.1.

\begin{remark}
As a direct application of our main result we note that any axially symmetric surface with Neumann boundary cannot have $H<0$ everywhere. Indeed, it follows from section~\ref{Section_Negative_MC} that no singularity can develop under mean curvature flow in the negative mean curvature setting. However, an enclosing cylinder of radius $y_{max}+1$ must collapse onto a line under mean curvature flow at time $T=(y_{max}+1)^2/2$. This yields a contradiction with the well-known barrier principle for mean curvature flow.
\end{remark}

\section*{Appendix: non-cylindrical maximum principle}
\label{Section_Max_Princ}
\noindent
In this section we state the maximum principle for non-cylindrical domains which was required for the proof of Theorem 1.1. In particular this extends work of Ecker \cite{KEB1} and Lumer \cite{GL872} to our setting. Note that in \cite{GL872} these are discussed in an operator theoretic setting. \\

\noindent
Let $\Lambda = M^n$. Let $V \subset \Lambda \times (0, T)$ be an open non-cylindrical domain. Let $\Lambda_t = \Lambda \times \{ t \} \, ,$ and for $t \neq 0 $ let  $V_t = \Lambda_t \cap V $, the cross sections of $V$ for constant $t$. Let $\overline{V}$ denote the closure of $V$ and $V_0 = \Lambda_0 \cap \overline{V} \, .$  The boundary of $V$ \, is $\partial V = \overline{V} \backslash V \, .$ The parabolic boundary is $\Gamma_V = \partial V \backslash \Lambda_T  \, .$ To describe the horizontal parts of the boundary of $V$ in the space-time diagram, we define the following: let $Z_t$ be the largest subset of $\Lambda_t \cap \partial V $ that is open in $\partial V$ and can be reached from ``below'' (with $t$ the vertical axis) in $V$ . Let $Z_V = \bigcup_{0 < t < T } Z_t $ and $\delta_V =  \Gamma_V \backslash Z_V  \, .$

\begin{proposition}\label{Prop_Max_Principles}{\emph{(Non-Cylindrical Maximum Principle)}}
Let $\left(M_t \right)_{t \in (0, T)}$ be a solution of the mean curvature flow \eqref{eq:int_2} consisting of hypersurfaces $M_t=\mathbf{x}_t(\Lambda)$, where $ \mathbf{x}_t = \mathbf{x}( \cdot, t) : \Lambda \times [0, T) \rightarrow \bigR^{n+1}$ and $\Lambda$ is compact. Suppose $f \in  C^{2,1}(V) \cap C(\overline{V}) $  satisfies an inequality of the form
$$ \left( \frac{d}{d t} - \Delta \right) f \leq \langle a , \nabla f \rangle \, ,
$$ 
where the Laplacian $\Delta$ and the gradient $\nabla$ are computed on the manifold $M_t$. For the vector field $a: V \rightarrow \bigR^{n+1}$ we only require that it is continuous in a neighbourhood of all maximum points of $f$ . Then
$$ \sup_V f \leq  \sup_{\Gamma_V} f \, ,
$$
for all $t \in [0, T)$.  \\
Assuming $f$ to have a positive supremum in $V$ then
$$ \sup_V f \leq \sup_{\delta_V} f \, ,
$$
for all $t \in [0, T) \, .$
\end{proposition}

\bibliographystyle{acm}
\bibliography{citations}

\vspace{.6cm}

\noindent
School of Mathematics, \\
Monash University, Australia  \\
john.head@monash.edu \\

\noindent
Department of Econometrics and Business Statistics,   \\
Monash University, Australia \\
sevvandi.kandanaarachchi@monash.edu

\end{document}